\documentclass[submission%
]{dmtcs}

\usepackage[latin1]{inputenc}
\usepackage{subfigure}

\usepackage[english]{babel}
\usepackage{amsmath,amssymb}
\usepackage{amsfonts}
\usepackage{asymptote}
\usepackage{array}
\usepackage{hyperref}
\usepackage{mathabx}
\usepackage{multirow}
\usepackage[colorinlistoftodos]{todonotes}
\usepackage[enableskew,vcentermath]{youngtab}
\usepackage[all]{xy}

\newtheorem{thm}{Theorem}[section]
\newtheorem{proposition}[thm]{Proposition}
\newtheorem{lemma}[thm]{Lemma}
\newtheorem{lemdef}[thm]{Lemma/Definition}
\newtheorem{corollary}[thm]{Corollary}
\newtheorem{conjecture}[thm]{Conjecture}

\newtheorem{definition}[thm]{Definition}
\newtheorem{remark}[thm]{Remark}
\newtheorem{example}[thm]{Example}

\newcommand{\newword}[1]{\textbf{\emph{#1}}}

\newcommand{\sh}{\mathrm{sh}}

\newcommand{\rectify}{\mathrm{rect}}

\newcommand{\esh}{\mathrm{esh}}
\newcommand{\lesh}{\mathrm{local\text{-}esh}}
\newcommand{\eset}{\varnothing}
\newcommand{\rect}{{\scalebox{.3}{\yng(3,3)}}}
\newcommand{\x}{\times}
\newcommand{\ybox}{\boxtimes}

\newcommand{\pieri}{{\bf \mathrm{Pieri}}}
\newcommand{\vertical}{{\bf \mathrm{Vert}}}

\newcommand{\horiz}{{\bf \mathrm{Horiz}}}
\newcommand{\cpieri}{{\bf \mathrm{CPieri}}}

\newcommand{\LR}{\mathrm{LR}}

\newcommand{\LRyb}{\LR(\alpha,{\scalebox{.5}{\yng(1)}},\beta, \gamma)}
\newcommand{\LRby}{\LR(\alpha,\beta,{\scalebox{.5}{\yng(1)}}, \gamma)}

\newcommand{\Kabg}{K(\gamma^c/\alpha; \beta)}

\author{Maria Monks Gillespie\addressmark{1}\thanks{Email: \email{monks@math.berkeley.edu}.  Supported by NSF GRFP and Hertz Foundation.}
  \and Jake Levinson\addressmark{2}\thanks{Email: \email{jakelev@umich.edu}. Supported by FRQNT.}}
\title[Monodromy of Schubert Curves]{Monodromy and $K$-theory of Schubert curves\\
  via generalized jeu de taquin}
\address{\addressmark{1}University of California, Berkeley, CA.\\
  \addressmark{2}University of Michigan, Ann Arbor, MI.}
\keywords{Young tableaux, monodromy, Schubert calculus, $K$-theory, osculating flag, jeu de taquin}
\received{???}
\revised{???}
\accepted{???}
\begin{document}
\maketitle
\begin{abstract}
\paragraph{Abstract.} 
We establish a combinatorial connection between the real geometry and the $K$-theory of complex \emph{Schubert curves} $S(\lambda_\bullet)$, which are one-dimensional Schubert problems defined with respect to flags osculating the rational normal curve.  In a previous paper, the second author showed that the real geometry of these curves is described by the orbits of a map $\omega$ on skew tableaux, defined as the commutator of jeu de taquin rectification and promotion.  In particular, the real locus of the Schubert curve is naturally a covering space of $\mathbb{RP}^1$, with $\omega$ as the monodromy operator.

We provide a fast, local algorithm for computing $\omega$ without rectifying the skew tableau, and show that certain steps in our algorithm are in bijective correspondence with Pechenik and Yong's \emph{genomic tableaux}, which enumerate the $K$-theoretic Littlewood-Richardson coefficient associated to the Schubert curve. Using this bijection, we give purely combinatorial proofs of several numerical results involving the $K$-theory and real geometry of $S(\lambda_\bullet)$.

\paragraph{R\'esum\'e.} Nous \'etablissons une connection entre la g\'eom\'etrie r\'eelle et la K-th\'eorie des \emph{courbes de Schubert} $\mathcal{S}(\lambda_\bullet)$. Ces derni\`eres sont des probl\`emes de Schubert, de dimension $1$, d\'efinies par rapport \`a des drapeaux tangents \`a la courbe rationelle normale. Le deuxi\`eme auteur a d\'emontr\'e auparavant que la g\'eom\'etrie de ces courbes est d\'ecrite par les orbites d`une transformation $\omega$ de tableaux de Young gauches : le commutateur de la rectification (au sens du \emph{jeu de taquin} de Sch\"utzenberger) et de la promotion. Les points r\'eels de $\mathcal{S}(\lambda_\bullet)$ forment alors un rev\^etement de $\mathbb{RP}^1$ avec $\omega$ comme op\'erateur de monodromie.

Nous introduisons un algorithme local et rapide qui permet de calculer $\omega$ sans devoir rectifier le tableau. Nous d\'emontrons ensuite que certaines \'etapes de l'algorithme sont en bijection avec les \emph{tableaux g\'enomiques} de Pechenik-Yong, lesquels \'enum\`erent le coefficient de Littlewood-Richardson K-th\'eorique associ\'e \`a $\mathcal{S}(\lambda_\bullet)$. Finalement, nous d\'emontrons de fa\c{c}on purement combinatoire certaines propri\'et\'es g\'eom\'etriques et K-th\'eoriques de $\mathcal{S}(\lambda_\bullet)$.
\end{abstract}

\section{Introduction}\label{sec:introduction}

This paper is an abridged version of the full paper \cite{bib:GillespieLevinson} by the authors.  We study the real and complex geometry of certain one-dimensional intersections $S$ of Schubert varieties defined with respect to `osculating' flags.  These curves are known to have smooth real points, which naturally cover the circle $\mathbb{RP}^1$; we give a new combinatorial rule, in terms of certain Young tableaux, for the monodromy operator on the fibers (a more complicated rule was given in \cite{bib:Levinson}). Our rule is fast and combinatorially `local', making it easier to count $\eta(S)$, the number of components of $S(\mathbb{R})$, which fully characterizes the real topology.  Moreover, our rule computes the class of $S$ in the $K$-theory of the Grassmannian: it explicitly produces Pechenik and Yong's \emph{genomic tableaux} \cite{bib:Pechenik}. This connection gives rise to purely combinatorial proofs of two known geometric relations between $\eta(S)$ and the Euler characteristic $\chi(\mathcal{O}_S)$, and also yields new facts about the real and complex geometry of $S$.

To define the curve $S$, recall first that the \emph{rational normal curve} is the image of the embedding $\mathbb{P}^1 \hookrightarrow \mathbb{P}^{n-1} = \mathbb{P}(\mathbb{C}^n)$ by the map
\[t \mapsto [1 : t : t^2 : \cdots : t^{n-1}].\]
Let $\mathcal{F}_t$ be the \emph{osculating} or \emph{maximally tangent flag} to this curve at $t \in \mathbb{P}^1$,  i.e. the complete flag in $\mathbb{C}^n$ formed by the iterated derivatives of this map. Let $G(k,\mathbb{C}^n)$ be the Grassmannian, and $\Omega(\lambda,\mathcal{F}_t)$ the Schubert variety for the condition $\lambda$ with respect to $\mathcal{F}_t$. The \newword{Schubert curve} is the intersection \[S = S(\lambda_1, \ldots, \lambda_r) = \Omega(\lambda_1, \mathcal{F}_{t_1}) \cap \cdots \cap \Omega(\lambda_r, \mathcal{F}_{t_r}),\]
where the osculation points $t_i$ are real numbers with $0 = t_1 < t_2 < \cdots < t_r = \infty$, and $\lambda_1,\ldots,\lambda_r$ are partitions for which $\sum |\lambda_i|=k(n-k)-1$.   For simplicity, we always consider intersections of only three Schubert varieties, though the results of this paper (in particular, Theorems \ref{thm:intro-2}, \ref{thm:MainResult2} and \ref{thm:intro-parity}) extend to the general case without difficulty.  With this in mind, we let $\alpha, \beta, \gamma$ be partitions with $|\alpha| + |\beta| + |\gamma| = k(n-k) - 1$, and we consider the Schubert curve
\[S(\alpha,\beta,\gamma) = \Omega(\alpha,\mathcal{F}_0) \cap \Omega(\beta,\mathcal{F}_1) \cap \Omega(\gamma,\mathcal{F}_\infty).\]

Schubert varieties with respect to such osculating flags have been studied extensively in the context of degenerations of curves \cite{bib:Chan} \cite{bib:EH86} \cite{bib:Oss06}, Schubert calculus and the Shapiro-Shapiro Conjecture \cite{bib:MTV09} \cite{bib:Pur13} \cite{bib:Sot10}, and the geometry of the moduli space $\overline{M_{0,r}}(\mathbb{R})$ \cite{bib:Speyer}. They satisfy unusually strong transversality properties, particularly when the osculation points $t$ are chosen to be real \cite{bib:EH86} \cite{bib:MTV09}; in particular, $S$ is known to be one-dimensional (if nonempty) and reduced \cite{bib:Levinson}. Moreover, intersections of such Schubert varieties in dimensions zero and one have been found to have remarkable topological descriptions in terms of Young tableau combinatorics. \cite{bib:Chan} \cite{bib:Levinson} \cite{bib:Pur10} \cite{bib:Speyer}

The Schubert curve is no exception: recent work (\cite{bib:Levinson}) has shown that its \emph{real} connected components can be described by combinatorial operations, related to jeu de taquin and Sch\"{u}tzenberger's promotion and evacuation,  on chains of skew Young tableaux. Recall that a skew semistandard Young tableau is \emph{Littlewood-Richardson} if its reading word is \emph{ballot}, meaning that every suffix of the reading word has partition content.   
  
 \begin{definition}\label{def:chains}
  We write $\mathrm{LR}(\lambda_1,\ldots,\lambda_r)$ to denote the set of sequences $(T_1, \ldots, T_r)$ of skew Littlewood-Richardson tableaux, filling a $k\times (n-k)$ rectangle, such that the shape of $T_i$ extends that of $T_{i-1}$ and $T_i$ has content $\lambda_i$ for all $i$. (The tableaux $T_1$ and $T_r$ are uniquely determined and may be omitted.)
\end{definition}

The theorem below describes the topology of $S(\alpha,\beta,\gamma)(\mathbb{R})$ in terms of tableaux:

\begin{thm}[\cite{bib:Levinson}, Corollary 4.9]\label{thm:intro-2}
There is a map $S \to \mathbb{P}^1$ that makes the real locus $S(\mathbb{R})$ a smooth covering of the circle $\mathbb{RP}^1$. The fibers over $0$ and $\infty$ are in canonical bijection with, respectively, $\LRyb$ and $\LRby$. Under this identification, the arcs of $S(\mathbb{R})$ covering $\mathbb{R}_-$ induce the \emph{jeu de taquin bijection}
\[\sh : \LRby \to \LRyb,\]
and the arcs covering $\mathbb{R}_+$ induce a different bijection $\esh$, called \emph{evacuation-shuffling}. The monodromy operator $\omega$ is, therefore, given by $\omega = \sh \circ \esh.$
\end{thm}

\vspace{-0.2cm}

\begin{figure}[hb]
\centering
\qquad \includegraphics[scale=0.57]{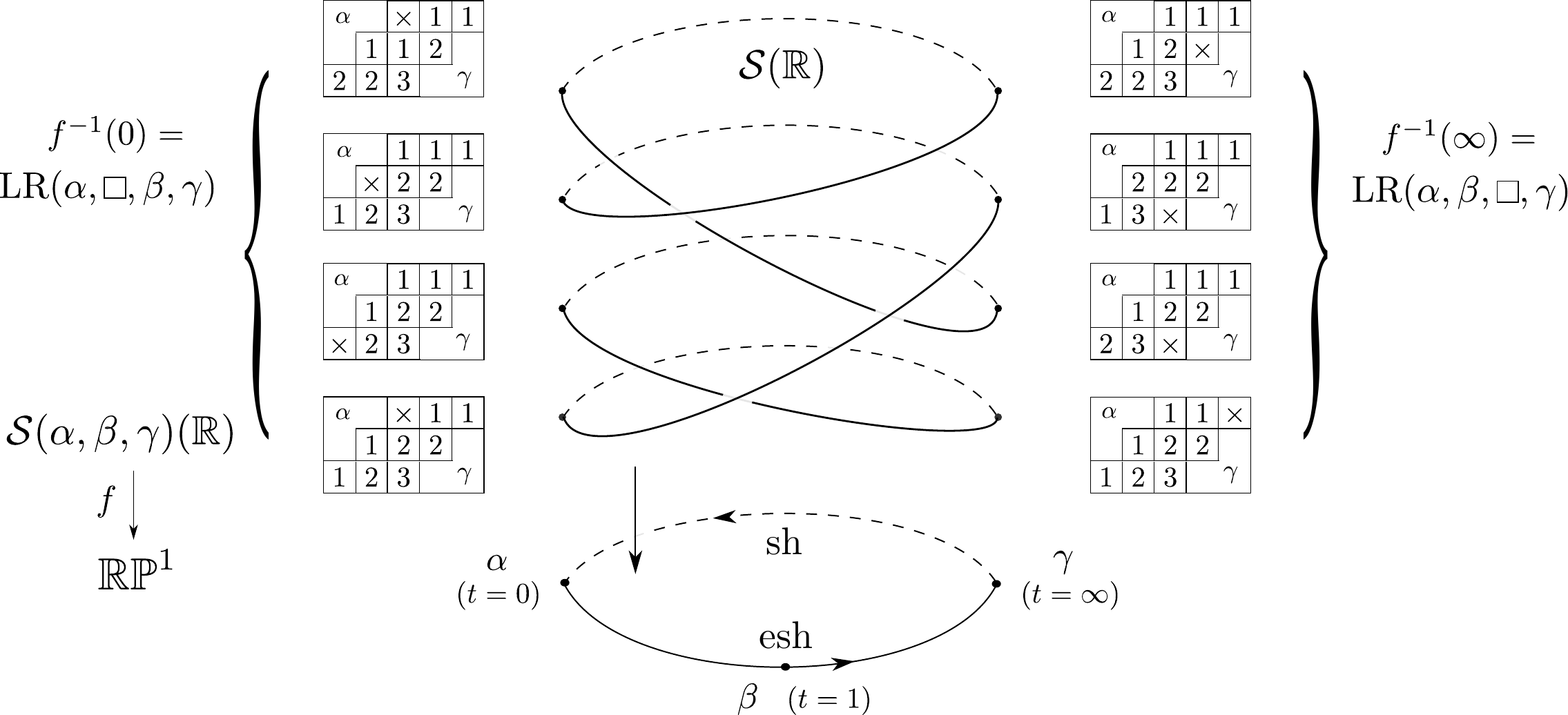}
\vspace{-0.2cm}
\caption{An example of the covering space of Theorem \ref{thm:intro-2}. The fibers over $0$ and $\infty$ are indexed by chains of tableaux, with $\ybox$ denoting the single box. The dashed arcs correspond to sliding the $\ybox$ using jeu de taquin. The monodromy operator is $\omega = \sh \circ \esh$.}
\label{fig:covering-space}
\end{figure}

\vspace{-0.25cm}

The operators $\esh$ and $\omega$ are our objects of study. In \cite{bib:Levinson}, the second author described $\esh$ as the conjugation of jeu de taquin \emph{promotion} by \emph{rectification} (see Section \ref{sec:background} for a precise definition). Variants of this operation have appeared elsewhere in \cite{bib:BerKir}, \cite{bib:HenrKam}, \cite{bib:KirBer}.


We prove two main theorems. The first is a shorter, `local' combinatorial description of the map $\esh$, which no longer requires rectifying or otherwise modifying the skew shape. We call our algorithm \emph{local evacuation shuffling}. Local evacuation-shuffling resembles jeu de taquin: it consists of successively moving the $\ybox$ through $T$ through a weakly increasing sequence of squares. Unlike JDT, the path is in general disconnected. (See Section \ref{sec:local-esh} for the definition, and Figure \ref{fig:antidiagonal} for a visual description of the path of the $\ybox$.)

\begin{figure}[t]
\centering
\includegraphics[scale=0.7]{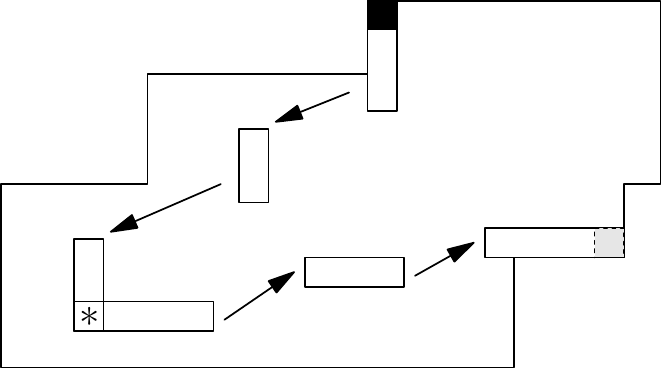}
\caption{The path of the $\ybox$ in a local evacuation-shuffle.  The black and gray squares are the initial and final locations of the $\ybox$; the algorithm switched from ``Phase 1'' to ``Phase 2'' at the square marked by a $*$. There is an \emph{antidiagonal symmetry}: the Phase 1 path forms a vertical strip, while the Phase 2 path forms a horizontal strip. We give a precise statement of this symmetry in the full paper.}
\label{fig:antidiagonal}
\end{figure}

\begin{thm}\label{thm:MainResult1}
  The map $\esh$ agrees with local evacuation shuffling. In particular, $\omega = \sh \circ \lesh$.
\end{thm}

Our second main result is related to K-theory $K(G(k,\mathbb{C}^n))$ and the orbit structure of $\omega$. We first recall a key consequence of Theorem \ref{thm:intro-2}:

\begin{proposition}[\cite{bib:Levinson}, Lemma 5.6] \label{prop:numerics}
Let $S$ have $\iota(S)$ irreducible components and let $S(\mathbb{R})$ have $\eta(S)$ connected components. Let $\chi(\mathcal{O}_S)$ be the holomorphic Euler characteristic. Then
\begin{align*}
\eta(S) &\geq \iota(S) \geq \chi(\mathcal{O}_S) \text{ and } \\
\eta(S) &= \chi(O_S) \pmod 2.
\end{align*}
\end{proposition}

We note that $\eta(S)$ is the number of orbits of $\omega$, viewed as a permutation of $\LRyb$. The numerical consequences above are most interesting in the context of $K$-theoretic Schubert calculus, which expresses $\chi(\mathcal{O}_S)$ in terms of both of ordinary and $K$-theoretic Young tableaux, namely
\[\chi(\mathcal{O}_S) = |\LRyb| - |\Kabg|.\]
See Section \ref{sec:K-theory} for the definition of $K(\gamma^c/\alpha; \beta)$, due to Pechenik-Yong \cite{bib:Pechenik}. In particular, we see that
\begin{align} \label{eqn:ktheory-ineq}
|\Kabg| &\geq |\LRby| - |\mathrm{orbits}(\omega)|, \text{ and} \\
\label{eqn:ktheory-mod2}
|\Kabg| &= |\LRby| - |\mathrm{orbits}(\omega)| \pmod 2.
\end{align}
Equivalently, we see that
\begin{align*}
|\Kabg| &\geq \mathrm{rlength}(\omega), \text{ and} \\
|\Kabg| &= \mathrm{sgn}(\omega) \pmod 2.
\end{align*}
where the sign is 0 or 1 and $\mathrm{rlength}(\omega)$ denotes the \newword{reflection length}, the minimum length of a factorization of $\omega$ as a product of transpositions (permutations consisting of a single $2$-cycle).

For the case where $\beta$ is a horizontal strip, a combinatorial interpretation of these facts was given in \cite{bib:Levinson}, indexing certain steps of an orbit by K-theoretic tableaux. Our second main result generalizes this combinatorial interpretation, showing that certain steps of local evacuation-shuffling correspond bijectively to the $K$-theoretic tableaux $\Kabg$:

\begin{thm}\label{thm:MainResult2}
As $T$ ranges over $\LRyb$, for either phase of the local description of $\esh(T)$, the gaps in the $\ybox$ path are in bijection with the set $K(\gamma^c/\alpha;\beta)$.
\end{thm}
Using the bijections of Theorem \ref{thm:MainResult2}, we give an independent, purely combinatorial proof of the relations \eqref{eqn:ktheory-ineq} and \eqref{eqn:ktheory-mod2}, by factoring $\omega$ into auxiliary operators $\omega_i$, which roughly correspond to the individual steps of local evacuation-shuffling, applied in isolation. If $\beta$ has $\ell(\beta)$ parts, we have the following:

\begin{thm}\label{thm:intro-parity} 
There is a factorization of $\omega$ as a composition $\omega_{\ell(\beta)} \cdots \omega_{1}$, such that each orbit $\mathcal{O}$ of each $\omega_i$ corresponds to $|\mathcal{O}|-1$ distinct K-theoretic tableaux.  Summing over the orbits yields \eqref{eqn:ktheory-ineq} and \eqref{eqn:ktheory-mod2}.
\end{thm}

In Section \ref{sec:local-esh}, we define $\lesh$ and sketch the proof that it agrees with $\esh$. Section \ref{sec:K-theory} contains the link to $K$-theory, and sketches of the proofs of Theorems \ref{thm:MainResult2} and \ref{thm:intro-parity}. Sections \ref{sec:constructions} and \ref{sec:conjectures} explore some consequences of the main results.

  

\section{Background}\label{sec:background}


\subsection{Tableaux and shuffling}
We refer to \cite{bib:Fulton} for the standard definitions of partitions, semistandard Young tableaux and jeu de taquin.  We briefly state some additional conventions that we will use.

Let $\lambda=(\lambda_1\ge \cdots \ge \lambda_k)$ be a partition.  We will refer to the partition $\lambda$ and its \newword{Young diagram} interchangeably throughout, using the English convention for Young diagrams. 
If $\mu$ is a partition with $\mu_i\le \lambda_i$ for all $i$, then the \newword{skew shape} $\lambda/\mu$ is the diagram formed by deleting the squares of $\mu$ from that of $\lambda$.  Its \newword{size}, written $|\lambda/\mu|$, is the number of squares that remain in the diagram. We will occasionally refer to (co-)corners of a skew shape. The \newword{inner} (respectively, \newword{outer}) \newword{corners} of $\lambda/\mu$ are the corners of $\lambda$ (respectively, the co-corners of $\mu$). These are the squares which, if deleted, leave a smaller skew shape. Similarly, the \newword{inner} (resp. \newword{outer}) \newword{co-corners} are the co-corners of $\lambda$ (resp. the corners of $\mu$): the exterior squares which can be added to obtain a larger skew shape.


We write $\rect=((n-k)^k)$ to denote a fixed rectangular shape of size $k\times (n-k)$, and we will always work with skew shapes that fit inside $\rect$.  The \newword{complementary} partition to $\lambda\subset \rect$, denoted $\lambda^c$, is the partition $(n-k-\lambda_k,n-k-\lambda_{k-1},\ldots,n-k-\lambda_1)$. 

Let $T$ be a semistandard Young tableau of shape $\lambda/\mu$.  The \newword{reading word} of $T$ is the sequence formed by reading the rows from bottom to top, and left to right within a row.  The \newword{suffix} of an entry $m$ of $T$ is the suffix of the reading word consisting of the letters \emph{strictly} after $m$.  The \newword{weak suffix} is the suffix including that letter and those after it.  A suffix is \newword{ballot for $(i,i+1)$} if it contains at least as many $i$'s as $i+1$'s, and is \newword{tied} if it has the same number of $i$'s as $i+1$'s. Finally, $T$ is \newword{ballot} or \newword{Littlewood-Richardson} (also known as \emph{Yamanouchi} or \emph{lattice}) if every weak suffix of its reading word is ballot for $(i,i+1)$, for all $i$.

Let $S,T$ be semistandard skew tableaux, such that the shape of $T$ \newword{extends} the shape of $S$, that is, $T$ can be formed by successively adding outer co-corners starting from $S$. Let $S'$ and $T'$ respectively be the tableaux formed by performing successive outward (resp. inward) jeu de taquin slides on $S$ (resp. $T$), using the entries of $T$ in ascending (resp. $S$, descending) order, and ordering equal entries of $T$ from left to right (resp. $S$, right to left).

\begin{definition}
The (jeu de taquin) \newword{shuffle} of $(S,T)$, denoted $\sh(S,T)$, is the pair of tableaux $(T',S')$.
\end{definition}

  The \newword{rectification} of a skew tableau $T$, denoted $\rectify(T)$, is the straight shape tableau formed by shuffling $T$ with any straight shape tableau $S$.  It is well known (often called the ``fundamental theorem of jeu de taquin'') that the result does not depend on $S$.

\subsection{Evacuation-shuffling and \texorpdfstring{$\omega$}{w}} \label{sec:shuffling-ops}

Note that a straight-shape or rotated straight-shape has only one Littlewood-Richardson tableau, so an element of $\LRyb$ is essentially a pair $(\ybox,T)$, with $T$ a Littlewood-Richardson tableau of content $\beta$, and $\ybox$ an inner co-corner of $T$, such that the shape of $\ybox \sqcup T$ is $\gamma^c/\alpha$. Computing $\esh(\ybox,T)$ consists of the following steps: \cite{bib:Levinson}
\begin{itemize}
  \item \textbf{Rectification.} Treat the $\ybox$ as having value $0$ and being part of a semistandard tableau $\widetilde{T}=\ybox \sqcup T$.  Rectify, i.e. shuffle $(S,\widetilde{T})$ to $(\widetilde{T}',S')$, where $S$ is an arbitrary straight-shape tableau.
  \item \textbf{Shuffling, or Promotion.}  (See \cite{bib:StanleyEC2} for the definition of promotion.)  Delete the $0$ of $\widetilde{T}'$ and rectify the remaining portion of $\widetilde{T}'$.  Label the resulting empty outer corner with $\ell(\beta)+1$.
  \item \textbf{Un-rectification.} Un-rectify, i.e. shuffle once more with $S'$.  Replace the entry $\ell(\beta)+1$ by $\ybox$.
\end{itemize}
Note that the promotion step is equivalent to shuffling the $\ybox$ past the rest of the rectified tableau. Thus, evacuation-shuffling corresponds to conjugating the ordinary jeu de taquin shuffle (on skew tableaux) by rectifying the tableau. This procedure outputs an element $(T',\ybox) \in \LRby$. Finally, the monodromy operator $\omega = \sh \circ \esh$ is the \emph{commutator} of rectification and shuffling.

\section{Local evacuation-shuffling} \label{sec:local-esh}

We will now define \newword{local evacuation-shuffling},
\[\lesh : \LRyb \to \LRby,\]
a local algorithm for computing $\esh$. The base case of our algorithm is the \newword{Pieri case}, where $\beta$ is a one-row partition. In this case, $\esh$ was computed in Theorem 5.10 of \cite{bib:Levinson}, and we recall it here.

\begin{thm}[Pieri case]\label{thm:Pieri-new}
Let $\beta$ be a one-row partition. Then $\esh(\ybox,T)$ exchanges $\ybox$ with the nearest $1 \in T$ \emph{prior to it} in reading order, if possible. If there is no such $1$, $\esh$ instead exchanges $\ybox$ with the \emph{last} $1$ in reading order (a \newword{special jump}).
\end{thm}

We give two examples, illustrating the possible actions of $\esh$ and the more familiar $\sh$. 

\begin{enumerate}
\item If the skew shape contains a (necessarily unique) vertical domino:
\[{\small \young(::\x11,:11,1)} \hspace{0.2cm}\stackrel{\xrightarrow{\esh}}{\xleftarrow[\,\sh\,]{}} \hspace{0.2cm}{\small \young(::111,:1\x,1)}\]
\item Otherwise, the action of $\esh \circ \sh$ cycles the $\ybox$ through the rows of $\gamma^c/\alpha$:
\[{\small\young(:::\x11,:11,1)} \hspace{0.2cm}\xrightarrow{\esh} \hspace{0.2cm}{\small\young(:::111,:1\x,1)}
\hspace{0.2cm}\xrightarrow{\sh} \hspace{0.2cm} {\small\young(:::111,:\x1,1)}\]
\[{\small\young(:::111,:11,\x)} \hspace{0.2cm}\xrightarrow{\esh} \hspace{0.2cm}{\small\young(:::11\x,:11,1)}
\hspace{0.2cm}\xrightarrow{\sh} \hspace{0.2cm} {\small\young(:::\x11,:11,1)}\]
\end{enumerate}

We refer the reader to \cite{bib:GillespieLevinson} or \cite{bib:Levinson} for two different proofs of this result.

\subsection{The algorithm}
We now state our new algorithm for evacuation-shuffling a box past an arbitrary ballot skew tableau.

\begin{definition}\label{def:algorithm}
Let $(\ybox, T) \in \LRyb$. We define the \emph{local evacuation-shuffle}, $\lesh(\ybox,T)$, by the following algorithm, starting at $i=1$.
  
  \begin{itemize}
    \item \textbf{Phase 1.} If the $\ybox$ does not precede all of the $i$'s in reading order, switch $\ybox$ with the nearest $i$ \emph{prior} to it in reading order. 
    Then increment $i$ by $1$ and repeat this step.  

    If, instead, the $\ybox$ precedes all of the $i$'s in reading order, go to Phase 2. 
    \item\textbf{Phase 2.} If the suffix from $\ybox$ is not tied for $(i,i+1)$, switch $\ybox$ with the nearest $i$ \emph{after it} in reading order. Repeat this process until the suffix becomes tied for $(i,i+1)$. Then increment $i$ by $1$ and repeat this step until $i=\ell(\beta)+1$.    
    
  \end{itemize}
\end{definition}
\begin{remark}
Phase 1 is identical to the Pieri case \emph{unless} the Pieri case calls for a special jump.
\end{remark}
In Phase 1, $\ybox$ moves down and to the left; in Phase 2, $\ybox$ instead moves to the right and up. (See Figure \ref{fig:antidiagonal}.)  We refer to the squares occupied by the box during the algorithm as the \newword{evacu-shuffle path}.

Note that in Phase 2, it is not obvious that we can find an $i$ after the $\ybox$ in reading order.  However, in \cite{bib:GillespieLevinson} we show the following lemma, which states that the tableau essentially remains semistandard and ballot at each step of the algorithm.  Consequently, the topmost $i$ is such a square.

\begin{lemma} \label{lem:pieri-jumps-yamanouchi}
Let $T_i$ be the tableau before the $i$-th step.  Then, omitting the $\ybox$, the rows (columns) of $T$ are weakly (strictly) increasing and the reading word of $T$ is ballot.
\end{lemma}

\begin{definition} 
  We use the following conventions for the $j$-th movement of the $\ybox$ in $\lesh$:
 
  $\pieri_j$ -- a \newword{regular Pieri jump}, a Phase 1 move in which the $\ybox$ moves down-and-left.
  
  $\vertical_j$ -- a \newword{vertical slide}, a Phase 1 move in which the $\ybox$ moves strictly down.
  
  $\cpieri_j$ -- a \newword{conjugate Pieri jump}, a Phase 2 move in which the $\ybox$ moves up-and-right.
  
  $\horiz_j$ -- a \newword{horizontal slide}, a Phase 2 move in which the $\ybox$ moves strictly right.
  
  We also say that $s$ is the \newword{transition step} if the algorithm switches to Phase 2 while $i = s$. If the algorithm remains in Phase 1 throughout, we say the transition step is $s=\ell(\beta)+1$.
\end{definition}

\begin{example}
  
   The diagram below demonstrates the $\lesh$ algorithm.
  \[\young(::\x 11,::122,::3,::4,23)\xrightarrow{\vertical_1}
    \young(::111,::\x 22,::3,::4,23)\xrightarrow{\pieri_2}
    \young(::111,::222,::3,::4,\x 3)\xrightarrow{\horiz_3}
    \young(::111,::222,::3,::4,3\x )\xrightarrow{\cpieri_4}
    \young(::111,::222,::3,::\x,34)\]
    \vspace{0.2cm}
\end{example}
 
\subsection{Proof of Theorem \ref{thm:MainResult1}} \label{sec:main-result}


In this section we outline the proof of the following:

\begin{thm}\label{thm:main-theorem}
  Local evacuation-shuffling agrees with evacuation-shuffling, that is, for any $(\ybox,T)$, \[\lesh(\ybox,T) = \esh(\ybox,T).\]
\end{thm}
 
\begin{figure}[t]
\begin{center}
  \includegraphics[height=2.5cm]{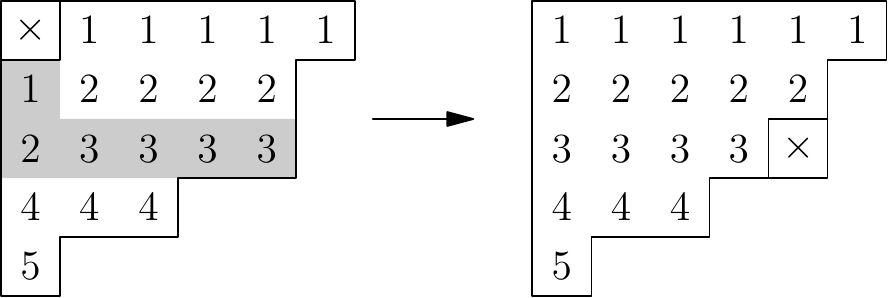} \qquad \hspace{1.3cm}\includegraphics[height=2.5cm]{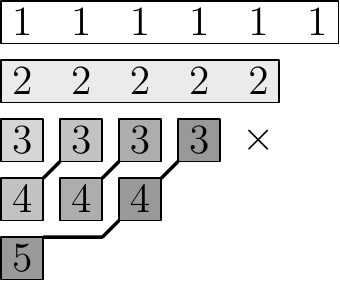}
\end{center}
\caption{An example of a rectified tableau $R$ with transition step $s=3$.  The promotion path of the box is down to row $s$ and then directly right. The corresponding \emph{$s$-decomposition} is shown at right.}
\label{fig:R-diagram} 
\end{figure}

The main idea is as follows.  In computing $\esh$, when we first rectify $(\ybox,T)$, we obtain a tableau $R$ of the form shown in Figure \ref{fig:R-diagram}.  In particular, the $\ybox$ is in the corner and the total shape of $\ybox \sqcup R$ is formed by adding an outer corner square to $\beta$ in some row $s$.  

When shuffling the $\ybox$ past $R$, the $\ybox$ follows a path directly down to row $s$ and then directly over to the end of row $s$, as shown.  It turns out that this corresponds to a more refined process in which we shuffle the $\ybox$ past rows $1,2,\ldots,s-1$, then shuffle it past the $\beta_s$ vertical strips formed by greedily taking vertical strips from the right of the remaining tableau.  We call this decomposition into horizontal and vertical strips the \newword{$s$-decomposition}, and we can similarly form the $s$-decomposition of the (unrectified, skew) tableau $T$ into horizontal and vertical strips.

Each step of Phase 1 of $\lesh$ corresponds to a single move of the $\ybox$ past a horizontal strip in the $s$-decomposition of $\beta$; the transition step then turns out to be $s$.  Using the \emph{antidiagonal symmetry} suggested by Figure \ref{fig:antidiagonal}, we show that the movements of the $\ybox$ during Phase 2 correspond similarly to shuffles past each of the $s$-decomposition's \emph{vertical} strips.

The complete proof given in \cite{bib:GillespieLevinson} uses the theory of \emph{dual equivalence classes} of tableaux (developed in \cite{bib:Haiman}), which are in bijection with Littlewood-Richardson tableaux.  The theory of dual equivalence allows us to use outwards rather than inwards rectification to compute $\esh$, which leads to the observed antidiagonal symmetry.

\section{Connections to K-theory}\label{sec:K-theory}

\subsection{Generating K-theoretic tableaux}

We recall the results we need on K-theory. The structure sheaves $\mathcal{O}_\lambda$ of Schubert varieties in $Gr(k,\mathbb{C}^n)$ form an additive basis for the K-theory ring $K(Gr(k,\mathbb{C}^n))$, with a product formula
\[[\mathcal{O}_\alpha] \cdot [\mathcal{O}_\beta] = \sum_{|\gamma^c| \geq |\alpha| + |\beta|} (-1)^{|\gamma^c| - |\alpha| - |\beta|}k_{\alpha \beta}^{\gamma^c} [\mathcal{O}_{\gamma^c}],\]
for certain nonnegative integer coefficients $k_{\alpha \beta}^{\gamma^c}$. 
%
%
In \cite{bib:Pechenik}, Pechenik and Yong introduced \newword{genomic tableaux} to enumerate $k_{\alpha \beta}^{\gamma^c}$, a `ballot semistandard' analog of Thomas and Yong's earlier theory \cite{bib:ThomasYong} of \emph{increasing tableaux}.  We state an equivalent characterization of genomic tableaux when $|\gamma^c| - |\alpha| - |\beta| = 1$.

\begin{lemdef} \label{lem:genomic-criterion}
Let $T$ be an (ordinary) semistandard tableau of shape $\gamma^c/\alpha$ and content equal to $\beta$ except for a single extra $i$. Let $\{\ybox_1,\ybox_2\}$ be a pair of squares of $T$. The data $(T, \{\ybox_1,\ybox_2\})$ corresponds to a \newword{ballot genomic tableau} if
\begin{itemize}
\item[(i)] The squares are non-adjacent and contain $i$,
\item[(ii)] There are no $i$'s between $\ybox_1$ and $\ybox_2$ in the reading word of $T$,
\item[(iii)] For $k = 1,2,$ the word obtained by deleting $\ybox_k$ from the reading word of $T$ is ballot.
\end{itemize}
We say that the \newword{K-theoretic content} is $\beta$. We write $\Kabg$ for the set of ballot genomic tableaux of shape $\gamma^c/\alpha$ and $K$-theoretic content $\beta$, and $\Kabg(i)$ for the tableaux with whose extra entry is $i$.

\end{lemdef}
\begin{thm}[\cite{bib:Pechenik}]
We have $k_{\alpha \beta}^{\gamma^c} = |\Kabg|$.
\end{thm}

We now show how $\lesh$ generates genomic tableaux. It turns out that each tableau in  $K(\gamma^c/\alpha;\beta)$ arises once during some step of Phase 1 and once during Phase 2, for some $T_1, T_2 \in \LRyb$. 
It is not hard to show, using Lemmas \ref{lem:pieri-jumps-yamanouchi} and \ref{lem:genomic-criterion}, that they arise from the non-horizontal and non-vertical jumps in the evacu-shuffle path, as follows.


\begin{thm} \label{thm:generating-ktheory}
Let $\ybox_1,\ybox_2$ be two successive non-adjacent squares in the evacu-shuffle path of $(\ybox,T)$ in which the $\ybox$ switches with an $i$.  Let $T_i$ be the tableau before this step in the path, with the $\ybox$ replaced by $i$.  Then the data $(T_i,\ybox_1,\ybox_2)$ corresponds to a ballot genomic tableau $T_K$, as in Lemma \ref{lem:genomic-criterion}.

Moreover, as $T$ ranges over $\LRyb$, every tableau $T_K \in K(\gamma^c/\alpha;\beta)(i)$ arises exactly once this way in Phase 1 and once more in Phase 2. This gives two bijections:

\begin{center} 
\begin{tabular}{l c l}
$\varphi_1$ : & $\big\{\pieri_i \text{ moves} \big\}$ & $\to\ K(\gamma^c/\alpha;\beta)(i),$ \\
$\varphi_2$ : & $\big\{\cpieri_j \text{ moves past } i\big\}$ & $\to\ K(\gamma^c/\alpha;\beta)(i)$.
\end{tabular}
\end{center}

\end{thm}




\subsection{The sign and reflection length of \texorpdfstring{$\omega$}{w}}

We now compute the sign of $\omega = \sh \circ \esh$, as a permutation of $\LRyb$, and the bound \eqref{eqn:ktheory-ineq} on its reflection length. We show:

\begin{thm} \label{thm:parity}
We have $\mathrm{rlength}(\omega) \leq |\Kabg|$ and $\mathrm{sgn}(\omega) = |\Kabg| \pmod 2$.
\end{thm}

\begin{lemma}
Let $X_i$ and $X'_i$, respectively, be the set of all tableaux arising in $\lesh$ and $\sh$, respectively, when the $\ybox$ is between the $(i-1)$-st and $i$-th horizontal strips.  Then $X_i = X'_i$.
\end{lemma}
\begin{proof}
Both sets consist of `punctured' semistandard tableaux of content $\beta$ and shape $\gamma^c / \alpha$, with ballot reading word, and where the $\ybox$ is between the $(i-1)$-st and $i$-th horizontal strips. (It is well-known that ballotness is preserved by jeu de taquin slides. Ballotness is also preserved during $\lesh$ by Lemma \ref{lem:pieri-jumps-yamanouchi}.) Both shuffling and evacuation-shuffling are invertible, so every such tableau arises in $X_i$ and $X'_i$.
\end{proof}

We have $X_1 = \LRyb$ and we write $X_{t+1} = \LRby$, where $t = \ell(\beta)$. For $1 \leq i \leq t$, we let $\ell_i : X_i \to X_{i+1}$ be the composition of the steps of $\lesh$ that switch the $\ybox$ with $i$'s. Let $s_i : X_{i+1} \to X_i$ be the jeu de taquin shuffle. We have the diagram

\[\xymatrix{
X_1 \ar@/_10pt/[r]_-{\ell_1} &
X_2 \ar@/_10pt/[r]_-{\ell_2} \ar@/_10pt/[l]_-{s_1}&
X_3 \ar@/_10pt/[r]_-{\ell_3} \ar@/_10pt/[l]_-{s_2} &
\cdots \ar@/_10pt/[r]_-{\ell_t} \ar@/_10pt/[l]_-{s_3} &
X_{t+1} \ar@/_10pt/[l]_-{s_t},
}\]
By definition, $\omega = \sh \circ \lesh = s_1 \circ \cdots \circ s_t \circ \ell_t \circ \cdots \circ \ell_1.$ Hence, for computing signs, we may rearrange:
\[\mathrm{sgn}(\omega) = \sum_{i=1}^t \mathrm{sgn}(s_i \circ \ell_i)\pmod{2}\]
The operators $\omega_i$ of Theorem \ref{thm:intro-parity} are the compositions $\omega_i=s_1\cdots s_{i-1} (s_i \ell_i) s_{i-1}\cdots s_1.$ We have $\omega = \omega_t \cdots \omega_1$, and since reflection length is subadditive,
\[\mathrm{rlength}(\omega) \leq \sum_{i=1}^t \mathrm{rlength}(\omega_i) = \sum_{i=1}^t \mathrm{rlength}(s_i \circ \ell_i).\]
We complete Theorem \ref{thm:parity} by describing the orbits of $s_i \circ \ell_i$, a computation interesting in its own right:
\begin{thm}\label{thm:ktheory-little-orbits}
Let $\mathrm{orb}_i$ be the set of orbits of $s_i \circ \ell_i$. Then:
\[\mathrm{rlength}(\omega_i) = \sum_{\mathcal{O} \in \mathrm{orb}_i} (|\mathcal{O}| - 1) = |K(\gamma^c/\alpha; \beta)(i)|.\]
\end{thm}
\begin{proof}[sketch]
We use the bijection $\varphi_1$ of Theorem \ref{thm:generating-ktheory} to generate genomic tableaux. Let $T \in X_i$.

First, suppose $\ell_i$ applies a Phase 1 vertical slide, or a sequence of Phase 2 moves consisting only of horizontal slides. Both of these steps are equivalent to jeu de taquin slides, and so $\ell_i(T) = s_i^{-1}(T)$. Thus $T$ is a fixed point of $\omega_i$ and does not contribute to the sum; it also does not generate a genomic tableau.

Otherwise, the orbit containing $T$ is similar in form to the Pieri case: all steps but one move the $\ybox$ downwards one row within the strip of $i$'s, generating one genomic tableau each. The last step begins in Phase 2 (it is a `special jump'), hence it does not generate a genomic tableau. (Unlike the Pieri case, the $\ybox$ moves downwards only until the $(i-1,i)$ suffix becomes tied, and `jumps' only far enough upwards to make the $(i,i+1)$ suffix tied.)

Thus each $\mathcal{O} \in \mathrm{orb}_i$ generates $|\mathcal{O}| - 1$ genomic tableaux. Every tableau of $K(\gamma^c/\alpha; \beta)(i)$ arises once in Phase 1, so we are done.
%
%
%
\end{proof}
%
%
%
%
\subsection{Fixed points of \texorpdfstring{$\omega$}{w}} \label{sec:omega-orbits}

We also characterize the fixed points of $\omega$:

\begin{proposition}\label{prop:fixed-points}
The fixed points of $\omega$ are the pairs $(\ybox,T)$ satisfying the (equivalent) conditions:
\begin{itemize}
\item[(i)] In the computation of $\lesh(\ybox, T)$, neither bijection $\varphi_1, \varphi_2$ generates a genomic tableau.
\item[(ii)] The evacu-shuffle path of the $\ybox$ is connected.
\end{itemize}
\end{proposition}

\begin{corollary}\label{cor:w=id}
Suppose $\omega$ acts on $\LRyb$ as the identity. Then $\Kabg = \eset$; it follows that the curve $S(\alpha, \beta, \gamma)$ is (over $\mathbb{C}$) a disjoint union of $\mathbb{P}^1$'s, and the map $S \to \mathbb{P}^1$ of Theorem \ref{thm:intro-2} is locally an isomorphism.
\end{corollary}

We note that in general, a morphism of real algebraic curves $C \to D$, inducing a covering map on real points, may have trivial \emph{real} monodromy but be algebraically nontrivial (i.e., not be a local isomorphism). Corollary \ref{cor:w=id} shows that this cannot occur for Schubert curves.


\section{Geometric constructions}\label{sec:constructions}


It is considerably easier to analyze the orbit structure of $\omega$, and, therefore, the geometric structure of the Schubert curve, using $\lesh$. As examples, we give two families of triples $(\alpha,\beta,\gamma)$ for which the Schubert curve $S(\alpha, \beta, \gamma)$ exhibits `extremal' numerical and geometrical properties.  See \cite{bib:GillespieLevinson} for full proofs, which rely on the relative simplicity of $\lesh$.

\begin{example}[Schubert curves of high genus]
Let $t \geq 2$ be a positive integer. Let
\[
\rect = (t+2)^{t+1}; \hspace{0.7cm}
\alpha = \gamma = (t,t-1,t-2,\ldots,2,1); \hspace{0.7cm}
\beta = (t+1,2,1^{t-2})
\]
so $\gamma^c/\alpha$ is a \emph{staircase-ribbon}. Then $\omega$ has \emph{only one orbit} on $\LRyb$. Hence, the Schubert curve $S_t \subset G(t+1,\mathbb{C}^{2t+3})$ is integral; moreover, its arithmetic genus is $g(S_t) = (t-1)(t-2)$.
\end{example}
In \cite{bib:Levinson}, the second author asked if Schubert curves are always smooth. K-theory does not in general detect singularities, but either possibility is interesting: that $S_t$ gives examples of singular Schubert curves for $t \gg 0$, or that it gives smooth Schubert curves of arbitrarily high (geometric) genus.

\begin{example}[Schubert curves with many connected components]
Let $t \geq 2$ be a positive integer. Let
\[
\rect = (t+1)^{t+1};\hspace{0.5cm}
\alpha = (t,t-1,\ldots, 2);\hspace{0.5cm}
\beta = (t,1,1); \hspace{0.5cm}
\gamma = (t+1,t,\ldots,3,2,2)
\]
Then $\omega$ \emph{acts as the identity} on $\LRyb$, which has $t-1$ elements. Consequently, the Schubert curve $S_t \subset G(t+1,\mathbb{C}^{2t+2})$ is a disjoint union of $t-1$ copies of $\mathbb{P}^1$.
\end{example}

\section{Conjectures}\label{sec:conjectures}


\begin{figure}[b]
\centering
\begin{tabular}{m{2cm} m{2cm} m{2cm} m{2cm}|c|c|c} 
\multicolumn{4}{c|}{Schubert problem} & $\mathcal{O}$ & $K_1(\mathcal{O})$ & $K_2(\mathcal{O})$ \\ \hline \vspace{0.1cm}
\multirow{3}{*}{$\alpha = {\tiny \yng(6,5,3,1)}$} &
\multirow{3}{*}{$\beta = {\tiny \yng(7,4,3,2)}$} &
\multirow{3}{*}{$\gamma = {\tiny \yng(5,5,4,2)}$} &
\multirow{3}{*}{$\rect = 6 \times 8$} &
38 & 52 & 51 \\
&&&& 23 & 31 & 28 \\
&&&& 10 & 9  & 13 \\ \hline \vspace{0.1cm}
\multirow{2}{*}{$\alpha = {\tiny \yng(2,2,1)}$} &
\multirow{2}{*}{$\beta = {\tiny \yng(3,1,1)}$} &
\multirow{2}{*}{$\gamma = {\tiny \yng(3,2)}$} &
\multirow{2}{*}{$\rect = 4 \times 4$} &
1 & 0 & 0 \\
&&&& 1 & 0 & 0
\end{tabular}
\caption{Examples of Schubert problems. For each problem, we list the size of each orbit $\mathcal{O}$ and the genomic tableaux $K_1(\mathcal{O}), K_2(\mathcal{O})$ generated in Phases 1 and 2. The second example demonstrates Corollary \ref{cor:w=id}.}
\label{fig:numerical-evidence}
\end{figure}


Numerical evidence, as in Figure \ref{fig:numerical-evidence}, suggests that the inequality \eqref{eqn:ktheory-ineq} in fact holds orbit-by-orbit, when the bijections $\varphi_1, \varphi_2$ of Theorem \ref{thm:generating-ktheory} are used to generate genomic tableaux:

\begin{conjecture} \label{conj:orbit-by-orbit}
Let $\mathcal{O} \subseteq \LRyb$ be an orbit of $\omega$. Let $K_1(\mathcal{O}), K_2(\mathcal{O})$ denote the sets of genomic tableaux occuring in this orbit in Phases 1 and 2 (via the bijections $\varphi_1, \varphi_2$). Then
\begin{equation} \label{eqn:orbit-by-orbit-ineq}
|K_i(\mathcal{O})| \geq |\mathcal{O}| - 1 \qquad (\text{for } i = 1, 2).
\end{equation}
\end{conjecture}
\noindent We have verified this conjecture in certain special cases:
\begin{thm}
Conjecture \ref{conj:orbit-by-orbit} holds in Phase 1 if $\beta$ has two rows, and in Phase 2 if $\beta$ has two columns.
\end{thm}
\noindent Also, it follows easily from Proposition \ref{prop:fixed-points} that \eqref{eqn:orbit-by-orbit-ineq} holds for orbits of size two (and is an equality for fixed points of $\omega$).

Finally, although we have only defined \emph{local} evacuation-shuffling for Littlewood-Richardson tableaux, the evacuation-shuffle $\esh$ is defined on \emph{all} tableaux $(\ybox,T)$ as the conjugation of shuffling by rectification. Our results do yield local algorithms for certain other Knuth classes of tableaux 
via straightforward alterations to $\lesh$. 
It would be interesting to understand the actions of $\esh$ and $\omega$ on semistandard tableaux in general, and to extend the connection to K-theoretic Schubert calculus. To be precise:

\begin{conjecture}
Let $T$ be \textbf{any} (semi)standard skew tableau and $\ybox$ an inner co-corner of $T$. There exists a local algorithm for computing $\esh(\ybox,T)$, which does not require rectifying the tableau, such that:
\begin{itemize}
\item[(i)] Each step consists of exchanging the $\ybox$ with an entry of $T$, of weakly increasing value.
\item[(ii)] The Knuth equivalence class of the word of $T$ (omitting $\ybox$) is preserved throughout the algorithm.
\item[(iii)] The algorithm specializes to jeu de taquin (if $T$ is of straight shape) and $\lesh$ (if $T$ is ballot).
\end{itemize}
Each step should correspond (by conjugating with rectification) to a jeu de taquin slide of $\ybox$ through the rectification $\rectify(\ybox,T)$. \end{conjecture}

For a straight-shape tableau $T$ that is \emph{not} highest-weight, 
it would be interesting to find an analog of the $s$-decomposition to describe the path of the $\ybox$, and to use it to give a local algorithm on any skew tableau $T'$ whose rectification is $T$. Finally, we ask how to compute $\esh(S,T)$ locally, where both $S$ and $T$ may have more than one box.

\acknowledgements
\label{sec:ack}
We thank David Speyer, Oliver Pechenik, and Mark Haiman for many helpful conversations.

\end{document}